\documentclass[a4paper,10pt
]{amsart}
\usepackage{amsmath,amssymb,amsthm}

\usepackage[alphabetic, abbrev]{amsrefs}
\usepackage{enumerate
}
\usepackage{hyperref}
\usepackage[all]{xy}
\newdir{ >}{{}*!/-5pt/@{>}} 
\SelectTips{cm}{} 

\numberwithin{equation}{section}
\newtheorem{theorem}[subsection]{Theorem}

\newtheorem{lemma}[subsection]{Lemma}
\newtheorem{proposition}[subsection]{Proposition}
\theoremstyle{definition}
\newtheorem{definition}[subsection]{Definition}
\newtheorem{remark}[subsection]{Remark}
\newtheorem{example}[subsection]{Example}

\newtheorem{construction}[subsection]{Construction}

\newcommand{\bS}{\mathbb{S}}

\newcommand{\cC}{\mathcal{C}}
\newcommand{\cD}{\mathcal{D}}

\newcommand{\cI}{\mathcal{I}}

\newcommand{\cN}{\mathcal{N}}

\newcommand{\cS}{\mathcal{S}}

\DeclareMathOperator*{\hocolim}{hocolim}
\DeclareMathOperator{\Map}{Map}

\DeclareMathOperator{\THH}{THH}

\DeclareMathOperator{\colim}{colim}

\newcommand{\sm}{\wedge}
\newcommand{\iso}{\cong}

\newcommand{\bld}[1]{{\mathbf{#1}}}

\newcommand{\cy}{\mathrm{cy}}
\newcommand{\op}{{\mathrm{op}}}

\newcommand{\Spsym}[1]{{\mathrm{Sp}^{\Sigma}_{#1}}}

\DeclareMathOperator{\concat}{\sqcup}

\newcommand{\arxivlink}[1]{\href{http://arxiv.org/abs/#1}{\texttt{arXiv:#1}}}

\usepackage[pdftex]{graphics,color} 
\usepackage{eso-pic}
\usepackage{scrtime}
\usepackage{ifdraft}
\ifdraft{%
\AddToShipoutPicture{\put(30,30){\resizebox{!}{.4cm}%
   {{{\color[gray]{0.8}\texttt{Preliminary draft, compiled \the\year-\the\month-\the\day~at \thistime}}}}}}
}{%
}

\ifdraft{%
  \newcommand{\Bemerkung}[1]{{\marginpar{\hspace{0.2\marginparwidth}\rule{0.6\marginparwidth}{0.75mm}\hspace{0.2\marginparwidth}}\noindent\bfseries[#1]}}
}{
  \newcommand{\Bemerkung}[1]{}
}

\title[\texorpdfstring{$\THH$}{THH} and the cyclic bar construction]{Topological Hochschild homology and the cyclic bar construction in symmetric spectra}
\author{Irakli Patchkoria} \address{Department of Mathematical Sciences,
University of Copenhagen,\newline
Universitetsparken~5,
2100 Copenhagen \O,
Denmark}
\email{irakli.p@math.ku.dk}

\author{Steffen Sagave} \address{
Radboud University Nijmegen, IMAPP, PO Box 9010, 6500 GL Nijmegen,\newline The Netherlands} \email{s.sagave@math.ru.nl}

\thanks{The first author was supported by the Danish National Research Foundation through the Centre for Symmetry and Deformation (DNRF92).}
\date{\today}

\begin{document}
\begin{abstract}
  The cyclic bar construction in symmetric spectra and B\"okstedt's
  original construction are two possible ways to define the
  topological Hochschild homology of a symmetric ring spectrum. In
  this short note we explain how to correct an error in Shipley's
  original comparison of these two approaches.
\end{abstract}
\maketitle

\section{Introduction}
When topological Hochschild homology was first introduced by Marcel
B\"okstedt in the unpublished manuscript~\cite{Boekstedt_THH}, a good
point set level model for the smash product of spectra was not yet known,
and $\THH$ was defined for \emph{functors with smash products}. One
can implement B\"okstedt's definition for a symmetric ring spectrum
$R$ by defining $\THH(R)$ to be the realization of the
simplicial symmetric spectrum
\begin{equation}\label{eq:Bokstedt-def-intro}
[k] \mapsto \THH_k(R) = \hocolim_{(n_0,\dots,n_k)\in\cI^{\times k+1}} \Omega^{n_0+ \dots + n_k}LF_0(R_{n_0} \sm \dots \sm R_{n_k}).
\end{equation}
Here $\cI$ is the category of finite sets and injections, $L$ is a level-fibrant replacement functor in symmetric spectra, and $F_0$ is the suspension spectrum functor. The functoriality of $\Omega^{n_0+ \dots + n_k}LF_0(R_{n_0} \sm \dots \sm R_{n_k})$ in the product category $ \cI^{\times k+1}$ comes from the structure maps of $R$, and the  simplicial structure maps of $[k] \mapsto \THH_k(R)$ arise from the multiplication and unit of $R$; see Construction~\ref{constr:DofRM} below. 

When viewing a symmetric ring spectrum $R$ as a monoid  with respect to the
smash product of symmetric spectra, one can also define its topological Hochschild homology as the realization of the cyclic bar construction  $[k] \mapsto B^{\mathrm{cy}}_k(R) = R^{\sm k+1}$. 

These two approaches are compared by Shipley in~\cite[Theorem 4.2.8]{Shipley_THH}. The first step in her argument is to construct a chain of stable equivalences relating $B^{\mathrm{cy}}_{\bullet}(R)$ and the simplicial object
\begin{equation}\label{eq:D-of-cyclic-bar-intro}
[k] \mapsto \textstyle\hocolim_{n \in\cI} \Omega^{n}LF_0(R^{\sm k+1})_n\ .
\end{equation}
Next Shipley shows that there are canonical stable equivalences relating the simplicial degree $[k]$ parts of~\eqref{eq:Bokstedt-def-intro} and \eqref{eq:D-of-cyclic-bar-intro}. However, it is erroneously stated in~\cite[Theorem 4.2.8]{Shipley_THH} that these maps form a morphism of simplicial objects. The problem is the compatibility with the last face map: The permutation of the $\cI$-coordinates that goes into the last face map of the simplicial object~\eqref{eq:Bokstedt-def-intro} has no counterpart in the simplicial structure of~\eqref{eq:D-of-cyclic-bar-intro}. We make this precise in Example~\ref{ex:coherence-problem} below. 

In Theorem~\ref{thm:main-result} below we provide a comparison of these two  definitions of $\THH$ that avoids this problem. Our strategy is to use a cofibrant replacement that allows to replace homotopy colimits by colimits in the critical part of the argument. 

\subsection{Conventions}
We assume familiarity with symmetric spectra and refer to~\cite{HSS}
and~\cite{Schwede_SymSp} as useful references for this topic. We will
often index spheres and the levels of symmetric spectra by finite sets
$\bld{n}=\{1,\dots,n\}$ rather than by natural numbers. This helps to
keep track of permutation actions.
\subsection{Acknowledgments}
After we first discovered the error in \cite[Theorem 4.2.8]{Shipley_THH} and later found the present workaround, Brooke Shipley encouraged us to prepare this note and make it available. We thank Stefan Schwede for comments on an earlier version of this note. We also thank the referees for useful comments.  

\section{Two models for \texorpdfstring{$\THH$}{THH}}

Let $\cI$ denote the category of finite sets $\bld{m}=\{1,\dots,m\}$, $m \geq 0$, and injective maps. It is symmetric strict monoidal under the ordered concatenation of ordered sets $\bld{m}\concat \bld{n} = \bld{m+n}$. The empty set $\bld{0}$ is the monoidal unit, and the block permutation $\tau_{(\bld{m},\bld{n})} \colon \bld{m}\concat\bld{n} \to \bld{n}\concat\bld{m}$ provides the symmetry isomorphism for $\concat$. 

As explained in~\cite[2.2.2.1 and 4.2.1.1]{Dundas_GMc_local}, applying the cyclic bar construction in the category of small categories $(\mathrm{cat})$ to $\cI$ provides a functor
\[
B^{\mathrm{cy}}_{\bullet}\cI \colon \Delta^{\op} \to (\mathrm{cat}), \qquad [k] \mapsto \cI^{\times k+1}. 
\] 
The simplicial face and degeneracy maps act by
\[
\begin{split}
d_i (\bld{n_0},\dots,\bld{n_k}) &= \begin{cases}(\bld{n_0},\dots,\bld{n_i}\concat\bld{n_{i+1}},\dots, \bld{n_k}) & 0\leq i < k \\ 
(\bld{n_k}\concat\bld{n_0},\dots,\bld{n_{k-1}}) & i = k
\end{cases} \quad\text{ and}\\
s_i(\bld{n_0},\dots,\bld{n_k}) &= (\bld{n_0},\dots,\bld{n_i},\bld{0},\bld{n_{i+1}},\dots,\bld{n_k}).
\end{split}
\]
Recall from~\cite[1.1 Definition]{Thomason-hocolim} that the Grothendieck construction on a functor $F\colon \cC \to (\mathrm{cat})$  is the category whose objects are the pairs $(C;X)$ with $C\in\mathrm{Ob}(\cC)$ and $X \in \mathrm{Ob}(F(C))$. A morphism $(C;X) \to (D;Y)$ is a pair $(\alpha;f)$ of morphisms $\alpha \colon C \to D$ in $\cC$ and $f\colon F(\alpha)(X) \to Y$ in $F(D)$. 
\begin{definition}\label{def:BcyI}
Let $B^{\mathrm{cy}}\cI$ be the Grothendieck construction of $B^{\mathrm{cy}}_{\bullet}\cI \colon \Delta^{\op} \to (\mathrm{cat})$. 
\end{definition}

Let $\Spsym{}$ be the category of symmetric spectra of simplicial sets
and let $L = \mathrm{Sing}|-| \colon \Spsym{} \to \Spsym{}$ be the
level fibrant replacement functor given by forming the singular
complex of the geometric realization in each level. Let $\cS_*$ be the
category of pointed simplicial sets, and let $F_0\colon \cS_* \to
\Spsym{}$ be the suspension spectrum functor.

The next construction is a reformulation of \cite[4.2]{Shipley_THH}
and \cite[4.2.2.3]{Dundas_GMc_local}. 
\begin{construction} \label{constr:DofRM}
Let $R$ be an associative symmetric ring spectrum and let $M$ be an $R$-bimodule. Let
\[\cD(R;M) \colon B^{\mathrm{cy}}\cI \to \Spsym{}\]
be the functor which is defined on objects by 
\[([k];\bld{n_0},\dots,\bld{n_k}) \mapsto \Map(S^{\bld{n_0}\concat \dots \concat \bld{n_k}},LF_0(M_{\bld{n_0}}\sm R_{\bld{n_1}} \sm \dots \sm R_{\bld{n_k}})).
\]
The morphisms in $\cI^{k+1}$ act via the symmetric group actions on
the levels of $R$ and $M$ and the structure maps of the spectra $R$
and $M$ \cite[Definition 4.2.2.1]{Dundas_GMc_local}. The morphisms in $\Delta$
act as in \cite[4.2.2.3]{Dundas_GMc_local}. For example, the last face
map gives rise to a morphism
\begin{equation}\label{eq:dk-in-BcyI}
([k];\bld{n_0},\dots,\bld{n_k}) \to ([k-1];\bld{n_k}\concat\bld{n_0},\dots,\bld{n_{k-1}})
\end{equation}
in $B^{\mathrm{cy}}\cI$ which acts by using the symmetry isomorphism that
moves $R_{\bld{n_k}}$ to the front of the iterated smash product, the multiplication $R_{\bld{n_k}} \sm M_{\bld{n_0}} \to M_{\bld{n_k}\concat\bld{n_0}}$, and the corresponding permutation of the sphere coordinates. The universal property of the Grothendieck construction and~\cite[Lemma 4.2.2.2]{Dundas_GMc_local} imply that this does indeed define a functor on $B^{\mathrm{cy}}\cI$. (The benefit of indexing $\cD(R;M)$ by $B^{\mathrm{cy}}\cI$ will become apparent in the next section.)
\end{construction}
Writing $\cD_k(R;M) = \cD(R;M)([k];-)\colon \cI^{\times k+1} \to \Spsym{}$,  we get a simplicial object in symmetric spectra \[[k] \mapsto \textstyle\hocolim_{\cI^{\times k+1}}\cD_k(R;M) =: \THH_k(R;M)\] where $\alpha \colon [k] \to [l]$ in $\Delta$ acts by
\[
\textstyle\hocolim_{\cI^{\times l+1}}\cD_l(R;M) \to \hocolim_{\cI^{\times l+1}}\cD_k(R;M)\circ \alpha^* \to \hocolim_{\cI^{\times k+1}}\cD_k(R;M). 
 \]
 Here the first map is induced by the functoriality of $\cD(R;M)$ in
 $B^{\mathrm{cy}}\cI$.  As discussed in the introduction, 
 the realization of $\THH_{\bullet}(R;M)$ is  B\"okstedt's model for $\THH$.  

Now let
\[
B^{\mathrm{cy}}_{\bullet}(R;M) \colon \Delta^{\op}\to \Spsym, \qquad [k] \mapsto M \sm R^{\sm k}
\]
be the cyclic bar construction in $(\Spsym{},\sm,\bS)$; see e.g.~\cite[4.2.1.1]{Dundas_GMc_local}. Let 
\[
\Omega^{\cI}_{\mathrm{Sp}} \colon \Spsym{} \to (\Spsym{})^{\cI}, \qquad (\Omega^{\cI}_{\mathrm{Sp}})(X)(\bld{m}) = \Map(S^{\bld{m}},L F_0 X_{\bld{m}})
\]
be the functor where isomorphisms $\bld{m}\to\bld{m}$ in $\cI$ act by conjugation and inclusions $\bld{m-1}\to\bld{m}$ act via the structure map of $X$ . Then $\hocolim_{\cI}\Omega^{\cI}_{\mathrm{Sp}}\colon \Spsym{}\to\Spsym{}$ is Shipley's detection functor \cite[Definition 3.1.1]{Shipley_THH}. It is shown in the first part of the proof of \cite[Theorem 4.2.8]{Shipley_THH} that there is a zig-zag of degreewise stable equivalences of simplicial objects in symmetric spectra relating  $\hocolim_{\cI}\Omega^{\cI}_{\mathrm{Sp}}B^{\cy}_{\bullet}(R;M)$ and $B^{\cy}_{\bullet}(R;M)$. To relate the former object to $\THH_{\bullet}(R;M)$, we note that there is a canonical map 
\begin{equation}\label{eq:can-maps-iterated-smash}
M_{\bld{n_0}}\sm R_{\bld{n_1}}\sm \dots\sm R_{\bld{n_k}} \to (M\sm R \dots \sm R)_{\bld{n_0}\concat\dots\concat \bld{n_k}}.
\end{equation}
(The map arises for example from identifying $(X\sm Y)_n$ for
symmetric spectra $X$ and $Y$ with $\colim_{\alpha \colon
  \bld{n_1}\concat\bld{n_2} \to \bld{n}} X_{n_1} \sm Y_{n_2} \sm
S^{\bld{n}\setminus \alpha(\bld{n_1}\concat \bld{n_2})}$, where the
colimit is taken over the comma category $-\concat- \downarrow \bld{n}$.)
Writing $\mu_{k+1}\colon \cI^{\times k+1} \to \cI$ for the iterated concatenation, the map \eqref{eq:can-maps-iterated-smash} induces a morphism of symmetric spectra 
\begin{multline}\label{eq:problematic-comparison}
\THH_k(R;M) = \textstyle\hocolim_{\cI^{\times k+1}}\cD_k(R;M)\\ \to
\textstyle\hocolim_{\cI^{\times k+1}} \mu_{k+1}^* \Omega^{\cI}_{\mathrm{Sp}}(M\sm R^{\sm k}) \to 
\textstyle\hocolim_{\cI}  \Omega^{\cI}_{\mathrm{Sp}}(M\sm R^{\sm k}).
\end{multline}
The problem with the proof of \cite[Theorem 4.2.8]{Shipley_THH} is that this map fails to provide a map of simplicial objects: 
\begin{example}\label{ex:coherence-problem}
We examine how the comparison maps in simplicial levels $0$ and $1$ interact with $d_1$. 
To simplify the exposition, we here ignore the suspension spectrum functor and the level fibrant replacement. Let $f\colon S^{\bld{n_0}\concat\bld{n_1}} \to M_{\bld{n_0}} \sm R_{\bld{n_1}}$ represent a  $0$-simplex in $\hocolim_{\cI^{\times 2}}\cD_1(R;M)$. First applying the map~\eqref{eq:problematic-comparison} and then the simplicial structure map $d_1$ of $B^{\cy}_{\bullet}(R;M)$ to $f$ amounts to forming
the composite 
\begin{equation}\label{eq:deg-2-1-diagram-upper}
S^{\bld{n_0}\concat\bld{n_1}} \xrightarrow{f} M_{\bld{n_0}} \sm R_{\bld{n_1}} \to (M\sm R)_{\bld{n_0}\concat\bld{n_1}} \xrightarrow{\tau} (R\sm M)_{\bld{n_0}\concat\bld{n_1}} \xrightarrow{\mu} M_{\bld{n_0}\concat\bld{n_1}}.
\end{equation}
Applying first $d_1$ and then the map~\eqref{eq:problematic-comparison} sends $f$ to the composite
\begin{equation}\label{eq:deg-2-1-diagram-lower}
S^{\bld{n_1}\concat\bld{n_0}} \xrightarrow{\tau_{(\bld{n_1},\bld{n_0})}} S^{\bld{n_0}\concat\bld{n_1}}\xrightarrow{f}M_{\bld{n_0}} \sm R_{\bld{n_1}}  \xrightarrow{\tau} R_{\bld{n_1}} \sm M_{\bld{n_0}}  \xrightarrow{\mu } M_{\bld{n_1}\concat\bld{n_0}}.
\end{equation} 
However, inspecting the commutative diagram
\[\xymatrix@-1pc{
S^{\bld{n_0}\concat\bld{n_1}}   \ar[r]^-{f} \ar[d]^{\tau_{(\bld{n_0},\bld{n_1})}}  &  M_{\bld{n_0}} \sm R_{\bld{n_1}}  \ar[r] \ar[d]^{\tau}  & (M\sm R)_{\bld{n_0}\concat\bld{n_1}}  \ar[r]^-{\tau}  & (R\sm M)_{\bld{n_0}\concat\bld{n_1}}  \ar[d]^-{\tau_{(\bld{n_0},\bld{n_1})}} \ar[r]^-{\mu} & M_{\bld{n_0}\concat\bld{n_1}} \ar[d]^-{\tau_{(\bld{n_0},\bld{n_1})}} \\
 S^{\bld{n_1}\concat\bld{n_0}}  \ar[r]&   R_{\bld{n_1}} \sm M_{\bld{n_0}}   \ar[rr]&& (R \sm M)_{\bld{n_1}\concat\bld{n_0}} \ar[r]^-{\mu} & M_{\bld{n_1}\concat\bld{n_0}}, \\ 
}\]
we deduce that the two maps~\eqref{eq:deg-2-1-diagram-upper} and~\eqref{eq:deg-2-1-diagram-lower} differ by the conjugation action of the block permutation $\tau_{(\bld{n_0},\bld{n_1})} \colon \bld{n_0}\concat\bld{n_1}\to \bld{n_1}\concat\bld{n_0}$. In fact, this is already indicated by the order of $\bld{n_0}$ and $\bld{n_1}$. Hence the points in $\hocolim_{\cI}\Omega^{\cI}_{\mathrm{Sp}}(M\sm R^{\sm k})$ represented by the two maps~\eqref{eq:deg-2-1-diagram-upper} and~\eqref{eq:deg-2-1-diagram-lower} do not coincide in general. Instead, they are only connected by the $1$-simplex represented by the morphism $\tau_{(\bld{n_0},\bld{n_1})}$ in $\cI$. 

This shows that the maps~\eqref{eq:problematic-comparison} fail to be compatible with the simplicial structure maps and do not induce a morphism on the realization. 
\end{example}

\section{Diagrams indexed by the cyclic bar construction on \texorpdfstring{$\cI$}{I}}
We now return to the setup of Definition~\ref{def:BcyI}. Let us for a
moment view the iterated concatenation in $\cI$ as a
functor \[\mu_{k+1}\colon \cI^{\times k+1}\to \Delta^{\op}\times
\cI,\qquad (\bld{n_0},\dots,\bld{n_k})\mapsto
([k],\bld{n_0}\concat\dots\concat\bld{n_k}).\] We claim that each
$\alpha \colon [k]\to[l]$ in $\Delta$ induces a natural
transformation 
\[\overline{\alpha} \colon \mu_{l+1} \Rightarrow \mu_{k+1} \circ \alpha^*\]
such that for $\beta \colon [l] \to [m]$, the following composition rule is
satisfied:
\[\overline{\beta\alpha} = (\overline{\alpha}\beta^*)(\overline{\beta}): \mu_{m+1} \Rightarrow \mu_{k+1} \circ \alpha^*\circ \beta^* =\mu_{k+1} \circ (\beta \alpha)^*\] To define $\overline{\alpha}$, we set $\overline{\alpha} = (\alpha,\mathrm{id})$ if $\alpha$ is a degeneracy map or a face map that is not equal to the last face map, and $\overline{\alpha} = (\alpha,\tau_{(\bld{n_0}\concat\dots\concat\bld{n_{k-1}},\bld{n_k})})$ if $\alpha$ is the last face map. Writing a general $\alpha$ as a composite of face and degeneracy maps, we can define $\overline{\alpha}$ by the above composition formula. This is well defined since our definition of $\overline{\alpha}$ for the face and degeneracy maps is compatible with the simplicial identities.  By the universal property of the Grothendieck construction~\cite[1.3.1
Proposition]{Thomason-hocolim}, we thus get a functor 
\begin{equation}\label{eq:mu-tw}
\mu^{\mathrm{tw}} \colon B^{\mathrm{cy}}\cI \to \Delta^{\op}\times \cI
\end{equation}
sending $([k];\bld{n_0},\dots,\bld{n_k})$ to $([k],\bld{n_0}\concat\dots\concat\bld{n_k})$. 

\begin{definition}
Let $E\colon \Delta^{\op}\to (\Spsym{})^{\cI}$ be a simplicial object in $\cI$-diagrams of symmetric spectra. Viewing it as a functor $E\colon \Delta^{\op}\times \cI\to \Spsym{}$, we let \[E^{\mathrm{tw}} \colon B^{\mathrm{cy}}\cI\to \Spsym{}\] be the composite $E \circ \mu^{\mathrm{tw}}$ of $E$ with the functor~\eqref{eq:mu-tw}.  
\end{definition}

We note that for $E\colon \Delta^{\op}\to (\Spsym{})^{\cI}$,  there is a canonical map
\begin{equation}\label{eq:mukplus1-for-E}
\hocolim_{\cI^{\times k+1}}E^{\mathrm{tw}}([k];-) \xrightarrow{\iso}\\ \hocolim_{\cI^{\times k+1}} \mu_{k+1}^*E([k],-) \to \hocolim_{\cI}E([k],-). 
\end{equation}
Analogous to Example~\ref{ex:coherence-problem}, the maps~\eqref{eq:mukplus1-for-E} do in general fail to be compatible with the last face map $d_k$ and thus do not assemble to a map of simplicial objects. However,  composing with the map from the homotopy colimit to the colimit, this can be resolved:

\begin{lemma}\label{lem:comp-after-colim}
The morphisms~\eqref{eq:mukplus1-for-E} become compatible with the simplicial structure maps after composing them with the canonical map $\hocolim_{\cI}\to\colim_{\cI}$. 
\end{lemma}
\begin{proof}
 Since the map from the homotopy colimit to the colimit is natural with respect to the change of the index category, it is sufficient to show that $\alpha \colon [k]\to[l]$ in $\Delta$  induces a commutative diagram 
\[\xymatrix@-1pc{
\colim_{\cI^{\times l+1}}E^{\mathrm{tw}}([l];-) \ar[r]\ar[d] &\colim_{\cI}E([l],-)\ar[d] \\
\colim_{\cI^{\times k+1}}E^{\mathrm{tw}}([k];-) \ar[r] &\colim_{\cI}E([k],-).
}\]
This is easy to verify for the degeneracy maps and all face maps but the last one. Let $\delta^{l}\colon [l-1]\to[l]$ be the last face map in $\Delta$, and let $x \in E^{\mathrm{tw}}([l];\bld{n_0},\dots,\bld{n_l})$ represent a simplex in one of the levels of the spectrum $\colim_{\cI^{\times l+1}}E^{\mathrm{tw}}([l];-)$. Then the composite through the upper right hand corner sends $x$ to the simplex represented by 
$(\delta^l)^*(x) \in E([l-1],\bld{n_0}\concat\dots\concat\bld{n_l}),$
while the other composite sends it to the simplex represented by
\[
(\tau_{(\bld{n_0}\concat\dots \concat\bld{n_{l-1}},\bld{n_l})})_*((\delta^l)^*(x)) \in E([l-1],\bld{n_l}\concat \bld{n_0}\concat\dots\concat\bld{n_{l-1}}).
\]
These represent the same simplex in the colimit. 
\end{proof}
We need some preparation to apply the lemma in a useful way.
\begin{definition}
Let $R$ be an associative symmetric ring spectrum and let $M$ be an $R$-bimodule. Then the \emph{twisted cyclic bar construction} is the $ B^{\mathrm{cy}}\cI$-diagram
\[
B^{\mathrm{cy}}(R;M)^{\mathrm{tw}} = \Omega^{\cI}_{\mathrm{Sp}}(B^{\mathrm{cy}}_{\bullet}(R;M))^{\mathrm{tw}} \colon B^{\mathrm{cy}}\cI\to \Spsym{}
\]
where $\Omega^{\cI}_{\mathrm{Sp}}$ and $B^{\mathrm{cy}}_{\bullet}(R;M)$ are as in the last section. 
\end{definition}

 Recall that a symmetric spectrum $X$ is \emph{semistable} if it admits a $\pi_*$-isomorphism to a symmetric $\Omega$-spectrum~\cite[5.6]{HSS}, and that it is \emph{flat} if it is $S$-cofibrant, i.e., cofibrant in the $S$-model structure developed in~\cite{Shipley_convenient}. We call a symmetric ring spectrum \emph{flat} if its underlying symmetric spectrum is flat. 
\begin{proposition}\label{prop:DRM-BcyRM}
  The canonical maps to the smash
  product~\eqref{eq:can-maps-iterated-smash} induce a natural
  transformation of $B^{\mathrm{cy}}\cI$-diagrams $\cD(R;M) \to
  B^{\mathrm{cy}}(R;M)^{\mathrm{tw}}$.   Fixing a simplicial degree $[k]$, the induced
  map 
\[
    \THH_k(R;M) = \textstyle\hocolim_{\cI^{\times k+1}}\cD_k(R;M) \to
    \hocolim_{\cI^{\times
        k+1}}B^{\mathrm{cy}}(R;M)^{\mathrm{tw}}([k];-)
\]
is a stable equivalence if $R$ is flat  and $R$ and $M$ are semistable. 
\end{proposition}
\begin{proof} It follows from the definitions that there is an induced map. 
The argument given in the proof of \cite[Theorem
  4.2.8]{Shipley_THH}, which is in turn based on \cite[Proposition
  4.2.3]{Shipley_THH}, shows that the composite of 
the map in the statement of the proposition 
with the map~\eqref{eq:mukplus1-for-E} for $E = \Omega^{\cI}_{\mathrm{Sp}}(B^{\mathrm{cy}}_{\bullet}(R;M))$
is a stable equivalence. Hence it is enough to show that \[\textstyle\hocolim_{\cI^{\times k+1}}\mu_{k+1}^*(\Omega^{\cI}_{\mathrm{Sp}}(B^{\mathrm{cy}}_{k}(R;M) ) \to \hocolim_{\cI}\Omega^{\cI}_{\mathrm{Sp}}(B^{\mathrm{cy}}_{k}(R;M))\]
is a stable equivalence. This follows from Lemma~\ref{lem:mul-hty-cofinal} and Lemma~\ref{lem:OmegaISp-semistable} below since by~\cite[4.10~Theorem]{Schwede_homotopy-groups},  our assumptions on $R$ and $M$ imply that $B^{\mathrm{cy}}_{k}(R;M)$ is semistable. 
\end{proof}

To apply Lemma~\ref{lem:comp-after-colim} to the cyclic bar construction, we employ the projective model structure on the diagram category $(\Spsym{})^{\Delta^{\op}\times \cI} = ((\Spsym{})^{\Delta^{\op}})^{\cI}$. This is the model structure where a natural transformation $f\colon X \to Y$ of $\Delta^{\op} \times \cI$-diagrams of symmetric spectra is a weak equivalence or fibration if $f([k],\bld{m})$ is a weak equivalence or fibration in the absolute projective stable model structure on $\Spsym{}$ for all objects $([k],\bld{m})$ of $\Delta^{\op}\times\cI$. Let
\begin{equation}\label{eq:cof-rep}
\xymatrix{C \ar@{->>}[r]^-{\sim} & \Omega^{\cI}_{\mathrm{Sp}}(B^{\mathrm{cy}}_{\bullet}(R;M))}
\end{equation}
be a cofibrant resolution in this model structure. Inspecting the generating cofibrations of the projective model structure on $(\Spsym{})^{\Delta^{\op}\times \cI}$, it follows that for each $[k]$, the map $\xymatrix{C([k],-) \ar@{->>}[r]^-{\sim} & \Omega^{\cI}_{\mathrm{Sp}}(B^{\mathrm{cy}}_{k}(R;M))}$
is a cofibrant replacement in $(\Spsym{})^{\cI}$. 
\begin{proposition}\label{prop:argument-replacement}
The cofibrant replacement and the natural map from the homotopy colimit to the colimit induce a zig-zag of stable equivalences
\begin{multline*}
\textstyle\hocolim_{\cI^{\times k+1}}B^{\mathrm{cy}}(R;M)^{\mathrm{tw}}([k];-) \xleftarrow{\sim}\hocolim_{\cI^{\times k+1}}C^{\mathrm{tw}}([k];-)\\ \xrightarrow{\sim} \colim_{\cI}C([k],-) \xleftarrow{\sim}\textstyle\hocolim_{\cI}C([k],-)  \xrightarrow{\sim} \textstyle\hocolim_{\cI}\Omega^{\cI}_{\mathrm{Sp}}(B^{\mathrm{cy}}_{k}(R;M))
\end{multline*}
that is compatible with the simplicial structure maps. 
\end{proposition}
We prove the proposition at the end of the section. 
\begin{theorem}\label{thm:main-result}
  Let $R$ be a flat symmetric ring spectrum, let $M$ be an
   $R$-bimodule spectrum, and assume that $R$ and $M$
  are semistable. Then there is a zig-zag of degreewise stable
  equivalences of simplicial objects relating
  $B^{\mathrm{cy}}_{\bullet}(R;M)$ and $\THH_{\bullet}(R;M)$. It
  induces a chain of stable equivalences after realization.
\end{theorem}
\begin{proof}
This follows by combining Propositions~\ref{prop:DRM-BcyRM} and~\ref{prop:argument-replacement} with the chain of degreewise stable equivalences relating $B^{\mathrm{cy}}_{\bullet}(R;M) $ and $\hocolim_{\cI}\Omega^{\cI}_{\mathrm{Sp}}(B^{\mathrm{cy}}_{\bullet}(R;M))$ from the proof of \cite[Theorem 4.2.8]{Shipley_THH}.
\end{proof}

\begin{remark}
One can use the argument outlined in \cite[Remark 4.2.10]{Shipley_THH} to get to a more general statement that avoids  the semistability assumption in the previous theorem. 
\end{remark}

\begin{remark}
When $M=R$, both $B^{\mathrm{cy}}_{\bullet}(R;M) $ and $\THH_{\bullet}(R;M)$ are cyclic objects, i.e., they extend to functor $\Lambda^{\op} \to \Spsym{}$ on Connes' cyclic category~$\Lambda$. Replacing $\Delta$  in our constructions by $\Lambda$ leads to a chain of stable equivalences relating these cyclic objects and therefore to a cyclic version of Theorem~\ref{thm:main-result}. After realization of the cyclic objects involved, we thus obtain a chain of stable equivalences of symmetric spectra with $S^1$-action relating $B^{\mathrm{cy}}(R)=|B^{\mathrm{cy}}_{\bullet}(R;R)|$ and $\THH(R) = |\THH_{\bullet}(R;R)|$. 

In view of the cyclotomic structure on the cyclic bar construction (of an orthogonal ring spectrum) recently established by Angeltveit et. al. \cite{Angeltveit-et-al_relative-cyclotomic}, one may ask if this zig-zag of stable equvialences induces a zig-zag of stable equivalences relating the resulting topological cyclic homology spectra. We don't have evidence that this follows directly from the present result. In fact, already the zig-zag of stable equivalences $B^{\mathrm{cy}}(R) \simeq \hocolim_{\cI}\Omega^{\cI}_{\mathrm{Sp}}(B^{\mathrm{cy}}(R))$  from the proof of \cite[Theorem 4.2.8]{Shipley_THH} does not appear to be well behaved with the passage to fixed points. 
\end{remark}
\subsection{Semistability results}
An  \emph{$\cI$-space} $X$ is a functor $X\colon \cI \to \cS_*$ from $\cI$ to based simplicial sets. Let $\cN \subset \cI$ be the subcategory given by the standard inclusions. A map of $\cI$-spaces $X \to Y$ is an \emph{$\cN$-equivalence} if the induced map of based homotopy colimits $\hocolim_{\cN}X \to \hocolim_{\cN}Y$ is a weak equivalence of spaces.  An $\cI$-space $X$ is \emph{semistable} if there is an $\cN$-equivalence $X \to Y$ with $Y$ homotopy constant, i.e.,  every $\alpha \colon \bld{m}\to\bld{n}$ induces a weak equivalence $\alpha_*\colon Y(\bld{m}) \to Y(\bld{n})$. This notion of semistability is studied in~\cite[2.5]{Sagave-S_group-compl} in the case of unbased $\cI$-spaces. 

\begin{lemma}\label{lem:mul-hty-cofinal}
Let $X\colon \cI \to \cS_*$ be a semistable $\cI$-space. Then the canonical map
\[
\textstyle\hocolim_{\cI^{\times k}}\mu_{k}^*(X) \to \hocolim_{\cI}X 
\]
is a weak equivalence. 
\end{lemma}

\begin{proof}
Suppose first that $X$ is homotopy constant. Then the canonical maps 
\[
X(\bld{n_1}\concat\dots \concat \bld{n_k}) \to \textstyle\hocolim_{\cI^{\times k}}\mu_{k}^*(X) \quad \text{and} \quad X(\bld{n_1}\concat\dots \concat \bld{n_k}) \to \hocolim_{\cI}X 
\]
are weak equivalences since the classifying spaces of $\cI$ and $\cI^{\times k}$ are contractible, see e.g.~\cite[Proposition 5.4]{Dugger_replacing}. This implies the result for a homotopy constant~$X$. For a semistable $X$, it is now sufficient to show that an $\cN$-equivalence $X \to Y$ induces weak equivalences
\[\textstyle\hocolim_{\cI}X \to \hocolim_{\cI}Y\quad\text{and}\quad\hocolim_{\cI^{\times k}}\mu_{k}^*(X) \to \hocolim_{\cI^{\times k}}\mu_{k}^*(Y).
\]
For the first map this follows from \cite[Proposition 2.2.9]{Shipley_THH}. The claim about the second map follows since there is a weak equivalence
\[
\hocolim_{\cI^{\times k}}\mu_k^*(X) \xrightarrow{\sim} \hocolim_{(\bld{n_1},\dots,\bld{n_{k-1}})\in\cI^{\times k-1}}\hocolim_{\bld{n_k}\in\cI} X((\bld{n_1}\concat \dots\concat \bld{n_{k-1}}) \concat\bld{n_k}))
\]
and restriction along $(\bld{n_1}\concat\dots\concat\bld{n_{k-1}})\concat - \colon \cI \to \cI$ preserves $\cN$-equivalences by a cofinality argument. 
\end{proof}
\begin{remark}
Since the classifying space of the undercategory $\bld{1}\downarrow \mu_2$ has two path components, the functor $\mu_{k}$ is in general not homotopy cofinal, and the last lemma does not hold without the semistability
hypothesis. 
\end{remark}
\begin{lemma}\label{lem:OmegaISp-semistable}
Let $E$ be a semistable symmetric spectrum. Then $\Omega^{\cI}_{\mathrm{Sp}}(E)$ is a semistable $\cI$-space in every spectrum degree. 
\end{lemma}
\begin{proof}
Let $E \to F$ be a $\pi_*$-isomorphism to a symmetric $\Omega$-spectrum $F$. Then in spectrum level $0$, the induced map $\Omega^{\cI}_{\mathrm{Sp}}(E) \to \Omega^{\cI}_{\mathrm{Sp}}(F)$ is  an $\cN$-equivalence to a homotopy constant $\cI$-space. The $\cI$-space in spectrum level $k > 0$ of $\Omega^{\cI}_{\mathrm{Sp}}(E)$ is isomorphic to the $\cI$-space in spectrum level $0$ of the $\cI$-symmetric spectrum $\Omega^{\cI}_{\mathrm{Sp}}(S^k \sm E)$ associated with the symmetric spectrum $S^k \sm E$. Since $S^k \sm E$ is semistable if $E$ is~\cite[4.6~Example]{Schwede_homotopy-groups}, the level $0$ case implies the general case. 
\end{proof}

\begin{proof}[Proof of Proposition~\ref{prop:argument-replacement}]
The compatibility with the simplicial structure maps follows from Lemma~\ref{lem:comp-after-colim}. It is clear that the first and the last map are stable equivalences. The third map is a stable equivalence because $C([k],-)$ is cofibrant in $(\Spsym{})^{\cI}$. Using once more that $\hocolim_{\cI}C([k],-) \to \colim_{\cI}C([k],-)$ is a stable equivalence reduces the claim about the second map to showing that $\hocolim_{\cI^{\times k+1}}C^{\mathrm{tw}}([k];-) \to \hocolim_{\cI}C([k],-)$ is a stable equivalence. In view of Lemma~\ref{lem:mul-hty-cofinal}, it is sufficient to show that $C([k],-)$ is semistable as an $\cI$-space in every spectrum degree. Since the cofibrant replacement~\eqref{eq:cof-rep} is an objectwise level acyclic fibration of symmetric spectra, this follows from  Lemma~\ref{lem:OmegaISp-semistable} since our assumptions on $R$ and $M$ imply that   $B^{\mathrm{cy}}_{k}(R;M)$ is semistable; see~\cite[4.10~Theorem]{Schwede_homotopy-groups}. 
\end{proof}



\end{document}